\documentclass{amsart}

\PassOptionsToPackage{obeyspaces}{url}
\usepackage[colorlinks=true,citecolor=blue,urlcolor=blue,linkcolor=blue,bookmarksopen=true]{hyperref}
\usepackage{amsthm}
\usepackage{comment}
\usepackage{slashbox}
\usepackage{etoolbox}
\patchcmd{\thebibliography}{\leftmargin\labelwidth}{\leftmargin\labelwidth\addtolength\itemsep{-0.1\baselineskip}}{}{}

\usepackage[nameinlink,sort]{cleveref}

\usepackage{graphicx} 
\usepackage{tikz}
\usetikzlibrary{calc}
\usepackage{graphicx}
\usepackage{booktabs}
\usepackage{amssymb}
\usepackage{amsmath}
\usepackage{amsthm, amsfonts}

\newtheorem{theorem}{Theorem}

\newtheorem{cor}[theorem]{Corollary}
\newtheorem{prop}[theorem]{Proposition}

\newtheorem{ques}{Question}

\newtheorem{conj}{Conjecture}

\newtheorem{exam}[theorem]{Example}

\newcommand*{\eqdef}{\stackrel{\mbox{\normalfont\tiny def}}{=}}   
\newcommand*{\abs}[1]{\lvert #1\rvert}                
\newcommand*{\R}{\mathbb R}

\DeclareMathOperator{\diam}{diam}
\DeclareMathOperator{\VR}{\mathcal{VR}}
\DeclareMathOperator{\Cl}{Cl}

\newcommand*{\rfacets}{F'}
\newcommand*{\dfacets}{f'}

\newcommand*{\level}[2]{\mathcal{L}^{#1}_{#2}}

\title{Facets in the Vietoris--Rips complexes of hypercubes}
\date{\today}

\author[J. Briggs]{Joseph Briggs}\address{Department of Mathematics and Statistics\\Auburn University\\Auburn\\AL~36849\\USA}
\email{jgb0059@auburn.edu}

\author[Z. Feng]{Ziqin Feng}
\address{Department of Mathematics and Statistics\\Auburn University\\Auburn\\AL~36849\\USA}
\email{zzf0006@auburn.edu}

\author[C. Wells]{Chris Wells}
\address{Department of Mathematics and Statistics\\Auburn University\\Auburn\\AL~36849\\USA}
\email{coc0014@auburn.edu}


\begin{document}
\maketitle
\begin{abstract}
    In this paper, we investigate the facets of the Vietoris--Rips complex $\VR(Q_n; r)$ where $Q_n$ denotes the $n$-dimensional hypercube.
    We are particularly interested in those facets which are somehow independent of the dimension $n$.
    Using Hadamard matrices, we prove that the number of different dimensions of such facets is a super-polynomial function of the scale $r$, assuming that $n$ is sufficiently large.
    We show also that the $(2r-1)$-th dimensional homology of the complex $\VR(Q_n; r)$ is non-trivial when $n$ is large enough, provided that the Hadamard matrix of order $2r$ exists.
\end{abstract}

\section{Introduction}
 Let $Q_n$ be the set $\{\pm 1\}^n$ equipped with Hamming distance $d$, i.e.\ for any $x, y\in Q_n$, $d(x, y)$ is the number of entries in which $x$ and $y$ differ.
The Vietoris--Rips complex $\mathcal{VR}(X;r)$ of a metric space $X$ with scale $r\geq 0$ is a simplicial complex with vertex set $X$, where a subset $\sigma\subseteq X$ is a simplex in $\mathcal{VR}(X;r)$ if and only if $\rho(x, y)\leq r$ for any pair $x, y\in \sigma$, where $\rho$ is the metric on $X$.
The hypercube graph is the graph with vertex set $Q_n$ such that any pair of vertices are adjacent if and only if their distance is $1$.
Therefore, the hypercube graph is same as  the Vietoris-Rips complex $\VR(Q_n; 1)$.
Recently, a lot of attention has been drawn to study the topological properties of Vietoris--Rips complexes of different metric spaces, for example, circles and ellipses (\cite{AA17, AAR19}), metric graphs (\cite{GGPSWWZ18}), geodesic spaces (\cite{ZV19, ZV20}), planar point sets (\cite{ACJS18}, \cite{CDEG10}), hypercube graphs (\cite{AA22, F23, FN24,  Shu22, AV24}), and more (\cite{MA13}, \cite{AAGGPS20}, \cite{AFK17}).

The study of the Vietoris--Rips complexes of hypercube graphs was initiated by Adamaszek and Adams in \cite{AA22}. Their main result in \cite{AA22} is that $\VR(Q_n;2)$ is homotopy equivalent to a wedge sum of $c_n$-many $3$-spheres where $c_n =\sum_{0\leq j<i<n}(j+1)(2^{n-2}-2^{i-1})$.
Motivated by the computed results in \cite{AA22} and using the results in \cite{FN24}, Feng~\cite{F23} studied these complexes with scale $3$ and proved that the complex $\VR(Q_n;3)$ is a wedge sum of $a_n$-many $4$-spheres and $b_n$-many $7$-spheres for $n\geq 5$ where $a_n= \sum_{i=4}^{n-1}2^{i-1}{i\choose 4}$ and $b_n=2^{n-4}{n\choose 4}$.
It is worth mentioning that Adams and Virk in \cite{AV24} discussed lower bounds of the rank of some non-trivial homologies with general scale $r$.
Shukla in \cite{Shu22} conjectured  that for $n\geq r+2$, $\tilde{H}_p(\VR(Q_n; r); \mathbb{Z})\neq 0$ if and only if $p=r+1$ or $2^r-1$.
This conjecture was invalidated by the computational findings, which shows that
\begin{align*}
    \tilde{H}_p(\VR(Q_6;4)) \cong
\begin{cases}
0 & \text{if } 0\leq p \leq 6 \\
(\mathbb{Z}/2)^{239} & \text{if } p=7  \\
 0  & \text{if } 8  \leq p \leq 14 \\
(\mathbb{Z}/2)^{14} & \text{if } p = 15\\
0& \text{if } p\geq 16.
\end{cases}
\end{align*}
We use the Easley Cluster at Auburn University (a system for high-performance and parallel computing), with up to $190$GB of memory, and the Risper software package \cite{Ba21}.
At a general scale $r$, the dimensions of non-trivial homologies in $\VR(Q_n; r)$ remain a mystery. The known results for $r=2$ and computed results for $r=4$ lead to the conjecture that for some scale $r$, the dimensions of non-trivial homologies of the complex $\VR(Q_n; r)$ should be in $\{2r-1, 2^r-1\}$. Another class of simplicial complexes is the Vietoris-Rips complex of the level set $\mathcal{L}_\ell^n$ which is closely related to the independence complex of Kneser graph.  The level set  $\mathcal{L}_\ell^{n}$ is the subset of $Q_n$ with all vertices containing exactly $\ell$ ones.  The computed results in Table~\ref{non_trivial} also using the Easley Cluster at Auburn University suggested the one of non-trivial homology dimensions in the complex $\VR(\mathcal{L}_\ell^n;r)$ should be $2r-2$ for some scale $r$.

The list of facets (maximal simplices) is necessary for the proof of almost all the known results in exploring the topological properties of the complex $\VR(Q_n; r)$.
As in \cite{AV24}, the $(2^r-1)$th homology generator in $VR(Q_n;r)$ is a natural extension of a facet with size $2^r$.
In this paper, we investigate different sizes of the facets in $\VR(Q_n;r)$ whose sizes do not depend on $n$.
With the help of Hadamard matrices, we construct the facets in $\VR(Q_n; r)$ whose size only depend on the scale $r$ and not on the dimension $n$.
We establish that the number of different dimensions of such facets is a super-polynomial function of the scale $r$, assuming that $n$ is sufficiently large.
Furthermore, we prove that the $(2r-1)$th homology of the complex $\VR(Q_n; r)$ and the $(2r-2)$th homology of $\VR(\mathcal{L}_\ell^n;r)$ is non-trivial for $\ell\geq r$ and $n$ being sufficiently large if  a Hadamard matrix of order $2r$ exists.
 \begin{table}\label{non_trivial}
 \caption{Ranks of non-trivial homologies in Vietoris--Rips complex $\VR(\mathcal{L}_3^n;4 )$ computed through Auburn University Easley Cluster.}
\label{KG_3_homology}
\begin{tabular}{ |c|p{0.8cm}|p{0.8cm}|p{0.8cm}|p{0.8cm}| }
 \hline
 \backslashbox[2mm]{homology}{$n$} & $6$ & $7$ & $8$& $9$   \\
  \hline
 $6$th-dim & $0$ &  29  & 233 & 1,052 \\
 \hline
 $9$th-dim & $1$ & $7$ &  $28$  & $84$  \\
 \hline
  other dim & $0$ & $0$ &  $0$  & $0$  \\
 \hline
 \end{tabular}
\end{table}

The paper is organized as follows. We start with some preliminaries in Section~\ref{pre}. In Section~\ref{sec:hada_simplex}, a class of facets in the complex $\VR(Q_n; r)$ is identified through Hadamard matrices. Section~\ref{sec:nontrivial_homo} is devoted to exploring the dimensions of non-trivial homology in the complexes $\VR(Q_n;r)$ and $\VR(\level n\ell; r)$ through the existence of Hadamard matrix of order $2r$.  The complex $\VR(\level n\ell; r)$ is closely related to the independence complex of the Kneser graph.
Then in Section~\ref{sec:facets}, we show that number of different dimensions in which $\VR(Q_n;r)$ contains a facet is a super-polynomial function of the scale $r$.
We conclude  with a list of open questions and conjectures in Section~\ref{sec:KG_questions}.

\section{Notation and Preliminaries}\label{pre}
\textbf{Simplicial complexes.}  A simplicial complex $K$ on a vertex set $V$ is a collection of non-empty subsets of $V$ such that: i) all singletons are in $K$; and ii) if $\sigma\in K$ and $\tau\subset \sigma$, then $\tau\in K$. For a complex $K$, we use $K^{(k)}$ to represent the $k$-skeleton of $K$, which is a subcomplex of $K$. For vertices $v_1, v_2, \ldots, v_k$ in a complex $K$,  if they span a simplex in $K$, then we denote the simplex to be $\{v_1, v_2, \ldots, v_k\}$. If $\sigma$ and $\tau$ are simplices in $K$ with $\sigma\subset \tau$, we say that $\sigma$ is a face of $\tau$.

A simplex which is maximal in $K$ is called a \emph{facet} of $K$.
\medskip

\textbf{Clique complex.} A simplicial complex $K$ is \emph{a clique complex} if  the following condition holds: a non-empty subset $\sigma$ of vertices is in $K$ given that the edge $\{v, w\}$ is in  $K$ for any pair $v, w\in \sigma$. For any graph $G=(V, E)$, we denote Cl$(G)$ to be the clique complex of $G$ whose vertex set is $V$ and a finite subset $\sigma\subset V$ is  a simplex in Cl$(G)$ if each pair of vertices in $\sigma$ forms an edge in $G$.
\medskip

\textbf{Diameter and rigidity.} Let $(X, d)$ be a finite metric space and $A$ be a subset of $X$. The \emph{diameter} of $A$ is $\diam(A)\eqdef\max_{x,y\in A}d(x,y)$.
We say that $A$ is \emph{maximally diameter-$r$} if $\diam(A)=r$ and $\diam(A\cup\{x\})>r$ for all $x\in X\setminus A$.
Finally, we say that $A$ is \emph{rigid} if $\max_{x\in A}d(x,a)=\diam(A)$ for every $a\in A$.

Note that the Vietoris--Rips complex over any metric space is a clique complex of its $1$-skeleton by the definition.
Any maximally diameter-$r$ subset $A$ of $X$ spans a facet in the Vietoris--Rips complex $\VR (X; r)$.
\medskip

 \textbf{Hadamard Matrix.} A \emph{Hadamard matrix} of order $n$ is a matrix $H\in\{\pm 1\}^{n\times n}$ satisfying
\[
    H^TH=HH^T=nI,
\]
where $I$ is the identity matrix. The famous Hadamard matrix conjecture asserts that there exists a Hadamard matrix of order $n$ whenever $n$ is a multiple of $4$.
Reversing the sign of a full row/column of a Hadamard matrix yields another Hadamard matrix; a Hadamard matrix is said to be \emph{normalized} if its first row and first column are all $1$'s.
\medskip

\textbf{Cross-polytopal complexes.} Let $G=(V, E)$ be the graph with vertex set $V=\{v_1, \dots, v_\ell, v_{-1},\dots,v_{-\ell}\}$ and $\{v_i, v_j\}\in E$ if and only if $i\neq -j$.
Then the clique complex $\Cl(G)$ is the boundary of a cross-polytope with $2\ell$ vertices which is homotopy to $S^{\ell-1}$. Furthermore, this complex has an $(\ell-1)$-dimensional cycle which generates $H_{\ell-1}(\Cl(G))$. We say this kind of complex to be \emph{cross-polytopal}.

The following result is straightforward to prove (see \cite{ZW24} for a proof).

\begin{prop}\label{cross_p_injective}
    Let $L$ be a cross-polytopal with $2\ell$ vertices.
    If $L$ is a subcomplex of $K$ and there is a simplex in $L$ which is maximal in $K$, then the inclusion map $\imath: L\rightarrow K$ induces an injective map in the $(\ell-1)$-dimensional homology and hence the $(\ell-1)$-dimensional homology of $K$ is non-trivial.
\end{prop}

\section{Hadamard matrices and Hadamard simplices in \texorpdfstring{$Q_n$}{the hypercube}}\label{sec:hada_simplex}

Recall that a simplex $\sigma\in\VR(Q_n;r)$ is said to be rigid if $\max_{y\in\sigma}d(x,y)=r$ for every $x\in\sigma$.
Write $\rfacets(n;r)$ to denote the set of all \emph{rigid} facets of $\VR(Q_n;r)$.
\medskip

A subset $A$ of $Q_n$ is said to be a \emph{Hadamard simplex} if $\abs A=n+1$ and $\langle x, y\rangle=-1$ for any $x, y\in A$ with $x\neq y$.
Here, $\langle\cdot,\cdot\rangle$ is the standard Euclidean inner product.
In other words, a Hadamard simplex is a regular simplex in $\R^n$ all of whose vertices live in $Q_n$.

The reason behind naming such a set after Hadamard is due to the following straightforward equivalence between Hadamard simplices and Hadamard matrices:
\begin{prop}\label{simplex_iff_basis}
    A set $A\subseteq Q_n$ is a Hadamard simplex if and only if $\{1\}\times A\subseteq Q_{n+1}$ is the set of columns of a Hadamard matrix of order $n+1$.
\end{prop}

In particular, a Hadamard simplex in $Q_n$ exists if and only if a Hadamard matrix of order $n+1$ exists.

Our first goal is to show that Hadamard simplices are rigid facets in the Vietoris--Rips complex of the appropriate scale.
To aid in this venture, we quickly record two observations about Hadamard simplices.
\begin{prop}\label{hada_obs}
    If $A\subseteq Q_n$ is a Hadamard simplex, then
    \[
        \sum_{x\in A}x=\mathbf 0,\qquad\text{and}\qquad \sum_{x\in A}xx^T=(n+1)I,
    \]
    where $\mathbf 0$ is the zero vector and $I$ is the identity matrix.
\end{prop}
\begin{proof}
    Both observations can be derived from Proposition~\ref{simplex_iff_basis}, but we will instead derive them directly from the definition and basic linear algebra facts.

    Let $M\in\{\pm 1\}^{n\times(n+1)}$ be any matrix whose set of columns is $A$.
    By definition, $M^TM=(n+1)I-J$ where $J$ is the all-ones matrix.
    In particular, $\mathbf 1\in\ker M^TM$ where $\mathbf 1$ is the all-ones vector and the eigenvalues of $M^TM$ are $0$ with multiplicity $1$ and $n+1$ with multiplicity $n$.

    Since $\ker M=\ker M^TM$,
    \[
        \mathbf 0=M\mathbf 1=\sum_{x\in A}x,
    \]
    which establishes the first part of the claim.
    Finally, since $M^TM$ and $MM^T$ have the same multiset of non-zero eigenvalues, $MM^T$ must have all of its $n$ eigenvalues equal to $n+1$.
    Since $MM^T$ is clearly diagonalizable, this forces
    \[
        (n+1)I=MM^T=\sum_{x\in A}xx^T.\qedhere
    \]
\end{proof}

As stated, we will prove that Hadamard simplices are rigid facets of the Vietoris--Rips complex of the appropriate scale.
To do so, it will be much more convenient to work with the inner product than with the Hamming distance; luckily, the two are equivalent in $Q_n$.
In particular, for $x,y\in Q_n$,
\begin{align}
    d(x,y) &={1\over 4}\lVert x-y\rVert_2^2={1\over 4}\bigl(\lVert x\rVert_2^2+\lVert y\rVert_2^2-2\langle x,y\rangle\bigr)={1\over 2}\bigl(n-\langle x,y\rangle\bigr)\nonumber \\
    \implies\quad \langle x,y\rangle &=n-2d(x,y).\label{hamming_innerprod}
\end{align}

\begin{theorem}\label{hada_simplex}
    If $A\subseteq Q_n$ is a Hadamard simplex, then $A\in\rfacets(n,{n+1\over 2})$.
\end{theorem}
\begin{proof}
    \eqref{hamming_innerprod} tells us that $d(x,y)\leq r$ if and only if $\langle x,y\rangle\geq n-2r$ for any $x,y\in Q_n$.
    It is therefore clear that $\diam(A)={n+1\over 2}$ and is rigid since $\langle x,y\rangle=-1$ for every distinct $x,y\in A$.

    In order to prove that $A$ is maximally diameter-${n+1\over 2}$, we must show that for every $y\in Q_n\setminus A$, there is some $x\in A$ for which $\langle x,y\rangle<-1$.

    Suppose for the sake of contradiction that $y\in Q_n\setminus A$ has the property that $\langle x,y\rangle\geq -1$ for every $x\in A$.
    Define
    \[
        E\eqdef\{x\in A:\langle x,y\rangle=-1\},
    \]
    and observe that the first part of Proposition~\ref{hada_obs} implies that
    \begin{equation}\label{eqn:innersum}
        \sum_{x\in A}\langle x,y\rangle=0\quad\implies\quad \abs E=\sum_{x\in A\setminus E}\langle x,y\rangle.
    \end{equation}

    Next, recall that $\langle x,y\rangle\geq -1$ for every $x\in A$; therefore, $\langle x,y\rangle\geq 0$ for every $x\in A\setminus E$ since all vectors have integer entries.
    From this observation, \eqref{eqn:innersum} and the second part of Proposition~\ref{hada_obs}, we bound
    \begin{align*}
        \abs E+\abs E^2 &= \abs E +\biggl(\sum_{x\in A\setminus E}\langle x,y\rangle\biggr)^2 \geq \abs E + \sum_{x\in A\setminus E}\langle x,y\rangle^2=\sum_{x\in A}\langle x,y\rangle^2\\
                        &= y^T\biggl(\sum_{x\in A}xx^T\biggr)y=y^T\bigl((n+1)I\bigr)y =(n+1)\langle y,y\rangle=n+n^2.
    \end{align*}
    Therefore, $\abs E\geq n$ and so either $\abs E=n$ or $\abs E=n+1$.

    If $\abs E=n+1$, then \eqref{eqn:innersum} implies that $n+1=0$; a contradiction.
    Otherwise, $\abs E=n$ and so $A\setminus E=\{x\}$.
    But then \eqref{eqn:innersum} implies that $n=\langle x,y\rangle$ and so $x=y$.
    This contradicts the fact that $y\notin A$ and concludes the proof of the theorem.
\end{proof}

\section{Non-trivial homologies of \texorpdfstring{$\VR(Q_n;r)$ and $\VR\bigl(\level n\ell;r\bigr)$}{the Vietoris--Rips complex}} \label{sec:nontrivial_homo}

\subsection{Hypercube}

Note that for positive integers $n_1,\dots,n_k$, we may naturally identify the product of cubes $\prod_{i=1}^k Q_{n_i}$ with the cube $Q_n$ where $n=\sum_{i=1}^k n_i$.

\begin{prop}\label{direct_sum}
    Fix positive integers $n_1,\dots,n_k$ and non-negative integers $r_1,\dots,r_k$.
    If $A_i\in\rfacets(n_i;r_i)$ for each $i$, then $\prod_{i=1}^k A_i\in\rfacets\bigl(\sum_{i=1}^k n_i;\sum_{i=1}^k r_i\bigr)$.
\end{prop}
\begin{proof}
    It suffices to prove the claim when $k=2$ since the full claim then follows by a routine induction.

    Observe that the Hamming distance is additive over the product; that is $d(x,y)=d(x_1,y_1)+d(x_2,y_2)$ when $x=(x_1,x_2)$ and $y=(y_1,y_2)$.
    It is therefore clear that $\diam(A_1\times A_2)=\diam(A_1)+\diam(A_2)$ and that $A_1\times A_2$ is rigid.

    To see why $A_1\times A_2$ is maximally diameter-$r$, fix any $y=(y_1,y_2)\in Q_{n_1}\times Q_{n_2}$ and suppose that $y\notin A_1\times A_2$.
    Without loss of generality, suppose that $y_1\notin A_1$.
    Since $A_1$ is maximally diameter-$r_1$, we can locate $a_1\in A_1$ such that $d(y_1,a_1)>r_1$.
    Then, since $A_2$ is maximally diameter-$r_2$ and also rigid, we can locate $a_2\in A_2$ such that $d(y_2,a_2)\geq r_2$.
    Thus, with $a=(a_1,a_2)\in A_1\times A_2$, we have $d(y,a)>r_1+r_2=r$, which demonstrates that $A_1\times A_2$ is maximally diameter-$r$.
\end{proof}

It is at this point we can see the importance of the assumption of rigidity.
\begin{cor}\label{rigid_extends}
    $X\in\rfacets(n;r)$ if and only if $X\times\{q\}\in\rfacets(n+m;r)$ for every $m\geq 1$ and every $q\in Q_m$.
\end{cor}
\begin{proof}
    Proposition~\ref{direct_sum} verifies the forward implication since $\{q\}\in\rfacets(m;0)$ for any $q\in Q_m$.

    For the other direction, the only situation which merits attention is when $X$ fails to be rigid.
    In this case, there is some $x\in X$ for which $d(x,y)\leq r-1$ for all $y\in X\setminus\{x\}$.
    But then, the point $(x,-1)$ is not an element of $X\times\{1\}\subseteq Q_{n+1}$, yet has distance at most $r$ from every point of $X\times\{1\}$, which demonstrates that $X\times\{1\}$ fails to be maximally diameter-$r$.
\end{proof}

We can now establish one additional non-trivial homology in $\VR(Q_n;r)$ for infinitely many values of $r$.

\begin{theorem}\label{nt_homology}
    If a Hadamard matrix of order $2r$ exists, then the $(2r-1)$-dimensional homology of $\VR(Q_n;r)$ is non-trivial for all $n\geq 2r-1$.
\end{theorem}
\begin{proof}
    By Proposition~\ref{simplex_iff_basis}, there is a Hadamard simplex $A\subseteq Q_{2r-1}$.
    Consider the set
    \[
        Y\eqdef(A\cup -A)\times\{\mathbf 1_{n-2r+1}\}\subseteq Q_n.
    \]
    It is easy to check that the subcomplex $\VR(Y;r)$ is cross-polytopal with $4r$ many vertices.
    Of course, Theorem~\ref{hada_simplex} and Corollary~\ref{rigid_extends} tell us that $A\times\{\mathbf 1_{n-2r+1}\}\subseteq Y$ is a facet of $\VR(Q_n;r)$ and so we conclude that the $(2r-1)$-dimensional homology of $\VR(Q_n;r)$ is non-trivial by applying Proposition~\ref{cross_p_injective}.
\end{proof}

\begin{exam} Fix $r=4$. The computation result shows that in $\VR(Q_6, 4)$ the rank of $7$th homology is $239$ and the rank of $15$th homology is $14$ using the coefficient group $\mathbb Z/2$. The computational result is used to provide a lower bound of the $7$th homology rank. However no information about the generators of $7$th homology is provided. Let $H$ be a normalized Hadamard matrix of order $4$. We find that removing the first two rows yields a facet in $\VR(Q_6; 4)$ which could be naturally extended to a homology generators of dimension $7$ as in the proof of Theorem~\ref{nt_homology}.

 Theorem 7.2 in \cite{AV24} guarantees that there are new topological features appearing for $n>6$, but it didn't provide a way to identify the new features. Using Theorem~\ref{nt_homology}, we are able to capture some of these new features and in fact we obtain that there is a new $7$th homology generator appearing at $n=7$. \end{exam}

In Table~\ref{tab:nt_homo}, we give a list of known non-trivial homologies for some $\VR(Q_n; r)$ from Theorem~\ref{nt_homology} and results in \cite{AV24}.
\begin{table}
\caption{Known non-trivial homologies in $\VR(Q_n; r)$ by Theorem~\ref{nt_homology}}
\label{tab:nt_homo}
\begin{tabular}{ |c|c|c|c| }
 \hline
 $n$ & $r$  & known non-trivial homology dimensions  \\

 \hline

 10 & 4 & 7, 15   \\
 \hline
 15 & 6 & 11, 63\\
 \hline
 20 & 8 & 15, 255\\
 \hline
 30 & 10 & 19, 1023 \\
 \hline

\end{tabular}

\end{table}

\subsection{Level-sets}

For a vector $x\in Q_n$, define $w(x)$ to be the number of $1$'s in $x$.
For non-negative integers $n,\ell$, the $\ell$th level of $Q_n$ is defined to be
\[
    \level n\ell\eqdef\{x\in Q_n: w(x)=\ell\}.
\]
Note that $\level n\ell$ can be naturally identified with the set of all $\ell$-sized subsets of $[n]$.
Observe that $\level{n_1}{\ell_1}\times\level{n_2}{\ell_2}$ can be naturally identified with a \emph{subset} of $\level{n_1+n_2}{\ell_1+\ell_2}$.

The Kneser graph, KG$(n, \ell)$, is the graph with vertices
being the collection of all $\ell$-subsets of $[n]$ and any pair of vertices being adjacent
if they have empty intersection. For any graph $G=(V, E)$, the independence complex of $G$, Ind$(G)$, is the simplicial complex with the simplices being the independent sets in the graph $G$. As noted in \cite{ZW24}, the independence complex of  KG$(n, \ell)$ is naturally a subcomplex of the Vietoris--Rips complex of the hypercube $Q_n$ with scale $2\ell-2$ in the following way.

For each $x\in Q_n$, define $v_x=\{i\in [n]: x(i)=1\}$. Then for each $x\in \level n\ell$, $v_x$ is a vertex in the Kneser graph, KG($n, \ell$). Also notice that $v_x\cap v_y\neq \emptyset$ if and only if $d(x, y)\leq 2\ell-2$. Therefore, the complex $\VR(\level n\ell; 2\ell-2)$ is identical to the independence complex Ind(KG$(n, \ell))$. Next we identify a new non-trivial homology in the complex $\VR(\level n\ell;r)$ if the Hadamard matrix of order $2r$ exists.

Consider a normalized Hadamard matrix $H$ of order $h$ and define $H^-$ to be the set of all columns of $H$ except for the first.
Since every element of $H^-$ is orthogonal to $\mathbf 1_h$, we see that $H^-\subseteq\level h{h/2}$.

\begin{prop}\label{hadamard_level}
    Suppose that $H$ is a normalized Hadamard matrix of order $2r$ for some $r\geq 4$.
    Fix integers $n,\ell$ with $\ell\geq r$ and $n\geq\ell+r$ and set
    \[
        A=H^-\times\{\mathbf 1_{\ell-r}\}\times \{-\mathbf 1_{n-\ell-r}\}.
    \]
    Then $A$ is a facet of $\VR(\level n\ell;r)$.
\end{prop}
\begin{proof}
    By construction, $H^-\subseteq\level{2r}r$, and so $A\subseteq\level n\ell$.
    We need to argue that $A$ is maximally diameter-$r$ in $\level n\ell$.

    It is clear that $A$ has diameter $r$, so fix some $x\in\level n\ell\setminus A$ and suppose for the sake of contradiction that $d(x,a)\leq r$ for all $a\in A$.
    Begin by decomposing $x=(x_1,x_2,x_3)\in Q_{2r}\times Q_{\ell-r}\times Q_{n-\ell-r}$.
    Since $d(x,a)\leq r$ for all $a\in A$, it must be the case that $d(x_1,h)\leq r$ for all $h\in H^-$.

    Consider first the situation when $w(x_1)\geq r$; so additionally $d(x_1,\mathbf 1_{2r})\leq r$.
    Thanks to Theorem~\ref{hada_simplex} and Corollary~\ref{rigid_extends}, we know that the set $H^-\cup\{\mathbf 1_{2r}\}$ is maximally diameter-$r$ and rigid in $Q_{2r}$ and so it must be the case that $x_1\in H^-\cup\{\mathbf 1_{2r}\}$.

    Suppose first that $x_1\in H^-$.
    Since $x\notin A$, it must be the case that $(x_2,x_3)\neq(\mathbf 1_{\ell-r},\mathbf 1_{n-\ell-r})$.
    Therefore, if $h\in H^-\setminus\{x_1\}$, then the point $a=(h,\mathbf 1_{\ell-r},\mathbf 1_{n-\ell-r})\in A$ has $d(x,a)\geq d(x_1,h)+1=r+1$; a contradiction.

    Next suppose that $x_1=\mathbf 1_{2r}$.
    Fix any $h\in H^-$ and consider the point $a=(h,\mathbf 1_{\ell-r},-\mathbf 1_{n-\ell-r})\in A$.
    Then
    \[
        d(x,a) =r+\bigl(\ell-r-w(x_2)\bigr)+w(x_3)=w(x)-w(x_2)+w(x_3)\geq w(x_1)=2r;
    \]
    another contradiction.

    We may therefore suppose that $w(x_1)=r-s$ where $s\geq 1$.
    Note that
    \[
        \ell=w(x)=w(x_1)+w(x_2)+w(x_3)\leq (r-s)+(\ell-r)+w(x_3)\implies w(x_3)\geq s.
    \]
    Then for any $a=(h,\mathbf 1_{\ell-r},-\mathbf 1_{n-\ell-r})\in A$,
    \begin{align*}
        r &\geq d(x,a)=d(x_1,h)+\bigl(\ell-r-w(x_2)\bigr)+w(x_3)\\
          &=d(x_1,h)+w(x_1)+2w(x_3)-r\geq d(x_1,h)+s.
    \end{align*}
    Therefore, $d(x_1,h)\leq r-s$.
    In particular, if $x_1'$ is formed by replacing any $s$ of the $-1$'s in $x_1$ by $1$'s, then $w(x_1')=r$ and $d(x_1',h)\leq r$ for all $h\in H^-$.
    By the work done earlier, the only way for this to be possible is if $x_1'\in H^-$.
    Of course, $\abs{H^-}=2r-1$, yet there are ${r+s\choose s}$ many possibilities for $x_1'$, and so we must have $s=1$.
    However, this implies that replacing any single $-1$ in $x_1$ by a $1$ yields an element of $H^-$, meaning that some pair of elements of $H^-$ are at distance $2$.
    This is impossible, however, since every pair of points of $H^-$ are at distance $r\geq 4$, which concludes the proof of the claim.
\end{proof}

\begin{theorem}\label{ntrvl_hmlg_kg}
    If a Hadamard matrix of order $2r$ exists and $r\geq 4$, then the $(2r-2)$-dimensional homology of $\VR(\level n\ell;r)$ is non-trivial for all integers $n,\ell$ satisfying $\ell\geq r$ and $n\geq\ell+r$.
\end{theorem}
\begin{proof}
    Let $H$ be a normalized Hadamard matrix of order $2r$ and consider the set
    \[
        Y\eqdef (H^-\cup -H^-)\times\{\mathbf 1_{\ell-r}\}\times\{-\mathbf 1_{n-\ell-r}\}.
    \]
    Since $H^-\subseteq\level{2r}r$, we have also that $-H^-\subseteq\level{2r}r$ and so $Y\subseteq\level n\ell$.
    It is then easy to check that the subcomplex $\VR(Y;r)$ is cross-polytopal with $2(2r-1)$ many vertices.
    Finally, Proposition~\ref{hadamard_level} tells us that $H^-\times\{\mathbf 1_{\ell-r}\}\times\{-\mathbf 1_{n-\ell-r}\}\subseteq Y$ is a facet of $\VR(\level n\ell;r)$, and so we conclude that the $(2r-2)$-dimensional homology of $\VR(\level n\ell;r)$ is non-trivial by applying Proposition~\ref{cross_p_injective}.
\end{proof}

The homotopy type of $\VR(\mathcal{L}_\ell^{n}; 2)$ is given in \cite{Bar13} and \cite{FN24}.  However, not much is known about the topological properties of the complexes  $\VR(\mathcal{L}_\ell^{n}; r)$ for $\ell\geq 3$, $r\geq 4$, and $n$ sufficiently  large. In \cite{ZW24}, the authors obtained a lower bound for the rank of $p$-dimensional homology of $\VR(\mathcal{L}_\ell^{n}; 2\ell-2)$ for $\ell\geq 3$ and $n$ being sufficient large, here $p={1\over 2}{2\ell \choose \ell}-1$.
The results in \cite{ZW24} show that the complex $\VR(\mathcal{L}^{n}_3; 4)$ for $n\geq 7$ has non-trivial homologies in dimensions $6$ and $9$ through the help of projective plane of order $2$. By Theorem~\ref{ntrvl_hmlg_kg}, the $6$th  homology of the complex $\VR(\mathcal{L}^{n}_4; 4)$ is non-trivial for $n\geq 8$.






\section{Rigid facets of \texorpdfstring{$\VR(Q_n;r)$}{the Vietoris--Rips complex}}\label{sec:facets}

Recall that $\rfacets(n;r)$ is the set of all rigid facets of $\VR(Q_n;r)$.
As discussed in the introduction, most properties of the Vietoris--Rips complex are understood by first locating an appropriate rigid simplex.
It is therefore natural to wonder how ``complicated'' the set $\rfacets(n;r)$ is.
While there are many ways to quantify ``complicated'', we will consider one of the most basic: how many different dimensions are seen in $\rfacets(n;r)$?
Explicitly, define
\[
    \dfacets(n;r)\eqdef\abs{\{\dim \sigma:\sigma\in \rfacets(n;r)\}}.
\]

By using Hadamard simplices, we will show that $\dfacets(n;r)$ is large.
To do so, we need the following estimate on the number of solutions to a linear inequality.

\begin{prop}\label{number_solu}
    For any positive numbers $s_1,\dots,s_n$, there are at least $\prod_{i=1}^n(s_i/i)$ many non-negative integer solutions to the inequality
    \[
        {x_1\over s_1}+\dots+{x_n\over s_n}\leq 1.
    \]
\end{prop}
\begin{proof}
    Let $N$ denote the number of non-negative integer solutions to the inequality.
    Observe that if $x_1,\dots,x_n$ is any non-negative (not-necessarily integral) solution to the inequality, then $\lfloor x_1\rfloor,\dots,\lfloor x_n\rfloor$ is a non-negative integer solution to the inequality.
    Therefore, if $X$ denotes the set of all non-negative solutions to the inequality, then $N$ is at least the volume of $X$.

    Finally, observe that $X$ is the convex hull of the points $\mathbf 0, s_1\mathbf e_1,\dots,s_n\mathbf e_n$ where $\mathbf e_i$ is the $i$th standard basis vector of $\R^n$.
    With this observation, it is a routine exercise to show that the volume of $X$ is precisely $\prod_{i=1}^n(s_i/i)$.
\end{proof}

\begin{theorem}\label{main}
    Fix a positive integer $\alpha$ and a sequence of positive integers $\alpha_1,\alpha_2,\ldots$ with $\alpha_i\leq\alpha$ for all $i$.
    For each $i$, set $h_i=p_i^{\alpha_i}$ where $p_i$ is the $i$th prime number.

    There is a universal constant $c$ such that if there exists a Hadamard matrix of order $4h_i$ for each $i$, then $\dfacets(n;r)\geq\exp\{cr^{1/(1+\alpha)}\}$ for all $n$ sufficiently large compared to $r$.
\end{theorem}
\begin{proof}
    Let $A_i$ be a Hadamard simplex of size $4h_i$  in $Q_{4h_i-1}$. So the diameter of the set $A_i$ is $2h_i$.
    Fix a value of $k$ to be determined later and consider any non-negative integers $x_1,\dots,x_k,r$ satisfying the inequality $\sum_{i=1}^n 2h_ix_i\leq r$. Set $r_0 =r-\sum_{i=1}^k 2h_ix_i$ and $n_0 =  r_0 +\sum_{i=1}^k x_i\cdot (4h_i-1)$. Fix $n>n_0$.

    By combining Theorem~\ref{hada_simplex} and Proposition~\ref{direct_sum}, one finds that
    \[
        A(x_1,\dots,x_k)\eqdef Q_{r_0}\times\prod_{i=1}^k A_i^{x_i}\times\{\textbf 1_{n-n_0}\}
    \]
    is a rigid facet of $\VR(Q_n;r)$ and has size
    \[
        s(x_1,\ldots, x_k)\eqdef\abs{A(x_1,\dots,x_k)}= 2^{r_0}\prod_{i=1}^k(4h_i)^{x_i}=2^{r_0+\sum_{i=1}^k 2x_i}\prod_{i=1}^k p_i^{\alpha_ix_i}.
    \]
    Since $p_1,\dots,p_k$ are distinct primes, $s(x_1,\ldots, x_k)$ is uniquely determined by the values of $x_1,\dots,x_k$ and so each of these rigid facets live in different dimensions.
    Therefore, $\dfacets(n;r)$ is at least the number of non-negative integer solutions to $\sum_{i=1}^k 2h_ix_i\leq r$.
    By Proposition~\ref{number_solu}, this number is at least $\prod_{i=1}^k{r\over 2ih_i}$.
    Therefore, we will have established the claim if we can show that $\prod_{i=1}^k{r\over 2ih_i}\geq\exp\{cr^{1/(1+\alpha)}\}$ for some appropriately chosen $k=k(r)$ and $c$.

    To do this, we need two asymptotic approximations:
    \[
        k!=\biggl({k\over e}\biggr)^{(1+o(1))k},\qquad\text{and}\qquad \prod_{i=1}^k p_i=k^{(1+o(1))k},
    \]
    where $o(1)\to 0$ as $k\to\infty$.
    The first approximation is due to Stirling and is well-known.
    The second approximation can be recovered from the asymptotic approximation of the Chebyshev function from {\cite[Page 27]{Hardy99}}, which states that $\sum_{i:\ p_i\leq x}\log p_i =\bigl(1+o(1)\bigr)x$, along with the prime number theorem, which states that $p_k=\bigl(1+o(1)\bigr)k\log k$.

    Combining these two approximations,
    \begin{align*}
        \prod_{i=1}^k{r\over 2ih_i}\geq\prod_{i=1}^k{r\over 2ip_i^\alpha}={r^k\over 2^k\cdot k!\cdot\bigl(\prod_{i=1}^kp_i\bigr)^\alpha}=\biggl({er\over 2k^{1+\alpha}}\biggr)^{(1+o(1))k},
    \end{align*}
    where $o(1)\to 0$ as $k\to\infty$.
    Selecting $k=\lfloor r^{1/(1+\alpha)}\rfloor$ then yields the claim.
\end{proof}

If Hadamard's conjecture were to hold, then we could take $\alpha=1$ in Theorem~\ref{main} and thus obtain $\dfacets(n;r)\geq\exp\{cr^{1/2}\}$.
Of course, the truth of Hadamard's conjecture is far from settled and so we can get only a weaker bound:
\begin{cor}
    There is a universal constant $c$ for which $\dfacets(n;r)\geq\exp\{cr^{1/5}\}$ for all $n$ sufficiently large compared to $r$.
\end{cor}
\begin{proof}
    Muzychuk and Xiang in \cite{MX06} constructed Hadamard matrices of order $4m^4$ for all odd integers $m$.
    We can therefore take $\alpha =4$ in Theorem~\ref{main}.
\end{proof}

\section{Concluding Remarks and Open Questions} \label{sec:KG_questions}

In the general, the following question remains widely open.

\begin{ques} Determine the formulas for the number of non-trivial homologies in the complexes $\VR(Q_n; r)$ and $\VR(\mathcal{L}^{n}_\ell; r)$.

Is it true that $\tilde{H}_{q}(\VR(Q_n; r))\neq 0$ iff $q\in \{2r-1, 2^r-1\}$ for an $r$ such that the Hadamard matrix of order $2r$ exists? \end{ques}

We proved that if the Hadamard matrix of order $2r$ exists then the $(2r-1)$th homology in $\VR(Q_n; r)$ is non-trivial for $n\geq 2r-1$. We do not know whether the converse holds. This can be considered to be a topological version of Hadamard's conjecture.

\begin{ques} Suppose that $\tilde{H}_{2r-1}(\VR(Q_n; r))\neq 0$ for all $n\geq 2r-1$. Is it true that the Hadamard matrix of $2r$ exists? \end{ques}
The authors in \cite{BG24} obtains a lower bound for the connectivity of $\VR{Q_n;r}$. It is worth to ask the following question about connectivity.
\begin{ques} Let $r$ be such that the Hadamard matrix of $2r$ exists. Is it true that the connectivity of  $\VR(Q_n; r)$ is $2r-2$ for $n\geq 2r-1$?  \end{ques}

Beyond the topological questions, we still do not understand most combinatorial properties of $\VR(Q_n;r)$.
The most interesting questions in this direction regard the sizes of the facets.

Let $b(n;r)$ denote the size of the largest facet in $\VR(Q_n;r)$; equivalently, $b(n;r)$ is the size of the largest diameter-$r$ subset of $Q_n$.
The value of $b(n;r)$ is actually well-known at this point and was originally established by Kleitman~\cite{Kleitman}, who showed that
\[
    b(n;r)=\begin{cases}
        \sum_{i=1}^{r/2}{n\choose i}, & \text{if $r$ is even},\\
        2\sum_{i=1}^{(r-1)/2}{n-1\choose i}, & \text{if $r$ is odd}.
    \end{cases}
\]
In the case when $r$ is even, the extremal examples are Hamming balls of radius $r/2$, and in the case when $r$ is odd, the extremal examples are the product of $Q_1$ with a Hamming ball of radius $(r-1)/2$ in $Q_{n-1}$.
Notice, however, that neither of these examples are rigid and so their maximality is tied to the ambient dimension of the host cube.
It is therefore interesting to define $b'(n;r)$ to be the size of the largest maximally diameter-$r$, rigid subset of $Q_n$ (which can exist only when $n\geq r$).
Here, an obvious candidate for an extremal set is the sub-cube $Q_r\times\{\mathbf 1_{n-r}\}$, which has size $2^r$.
\begin{conj}
    $b'(n;r)=2^r$ whenever $n\geq r$.
\end{conj}

Next, let $s(n;r)$ denote the size of the smallest facet in $\VR(Q_n;r)$.
We showed in Theorem~\ref{hada_simplex} that $s(n;r)\leq 2r$ whenever a Hadamard matrix of order $2r$ exists and $n\geq 2r-1$.
Unfortunately, we have not had any success in proving a matching lower bound.
It is tempting to conjecture that $s(n;r)=2r$ whenever a Hadamard matrix of order $2r$ exists and $n$ is large enough compared to $r$, but we would be satisfied with the following conjecture:
\begin{conj}
    There is a universal constant $c$ such that $s(n;r)\geq cr$ for all $n$ sufficiently large compared to $r$.
\end{conj}
One could, of course, ask also about $s'(n;r)$, the size of the smallest rigid facet in $\VR(Q_n;r)$, but we believe that $s'(n;r)=s(n;r)$ whenever $n$ is sufficiently large compared to $r$.

Finally, recall from Section~\ref{sec:facets} that $\dfacets(n;r)$ denotes the number of dimensions in which $\VR(Q_n;r)$ has a rigid facet.
If $b'(n;r)=2^r$ as we believe, then trivially $f(n;r)\leq 2^r$.
We gave a construction in Theorem~\ref{main} that implies that $\dfacets(n;r)\geq\exp\{cr^{1/5}\}$ whenever $n$ is sufficiently large.
While this bound is super-polynomial, it is still sub-exponential.
\begin{ques}
    Is $\dfacets(n;r)$ exponential in $r$ when $n$ is sufficiently large?
\end{ques}

One final question about the facets of $\VR(Q_n;r)$ relates directly to our construction from Hadamard matrices.
By combining Theorem~\ref{hada_simplex} and Corollary~\ref{rigid_extends}, if $H$ is a normalized Hadamard matrix of order $n$, then the columns of $H$ form a diametrically-maximal subset of $Q_n$.
Beyond this, as shown in Theorem~\ref{hada_simplex}, if $H_{-1}$ is the matrix formed by deleting the first row of $H$, then the columns of $H_{-1}$ are a diametrically-maximal subset of $Q_{n-1}$.
Our experiments seem to suggest this phenomenon continues when deleting more rows from $H$.
For a non-negative integer $r$ and a matrix $A$, let $A_{-r}$ denote the submatrix formed by deleting the first $r$ rows from $A$.
\begin{conj}
    For any non-negative integer $r$, there is a number $n_0$ such that whenever $H$ is a normalized Hadamard matrix of order $n\geq n_0$, then the columns of $H_{-r}$ form a diametrically-maximal subset of $Q_{n-r}$.
\end{conj}

\textbf{Acknowledgments}

We would like to thank Henry Adams for his helpful suggestions. This work was completed in part with resources provided by the Auburn University Easley Cluster.

\end{document}